\documentclass[12pt,reqno]{amsart}

\usepackage{epsfig}
\usepackage{color}
\usepackage{amsmath, amsthm, amsfonts, amssymb, mathrsfs}
\usepackage{graphicx}
\usepackage{enumerate}

\numberwithin{equation}{section}

\newcommand{\wh}{\widehat}

\newtheorem{theo}{{\sc \bf Theorem}}[section]
\newtheorem{cor}[theo]{{\sc \bf Corollary}}
\newtheorem{lem}[theo]{{\sc \bf Lemma}}

\theoremstyle{definition}

\newenvironment{rem}{\medskip\noindent{\bf Remark:\/} }{\medskip}

\begin{document}

\title{Spectral Triples for nonarchimedean local fields}

\author{Slawomir Klimek}
\address{Department of Mathematical Sciences,
Indiana University-Purdue University Indianapolis,
402 N. Blackford St., Indianapolis, IN 46202, U.S.A.}
\email{sklimek@math.iupui.edu}

\author{Sumedha Rathnayake}
\address{Department of Mathematics,
University of Michigan,
530 Church St., Ann arbor, MI 48109, U.S.A.}
\email{sumedhar@umich.edu}

\author{Kaoru Sakai}
\address{Department of Mathematical Sciences,
Indiana University-Purdue University Indianapolis,
402 N. Blackford St., Indianapolis, IN 46202, U.S.A.}
\email{ksakai@iupui.edu }

\date{\today}

\begin{abstract}
Using associated trees, we construct a spectral triple for the C$^*$-algebra of continuous functions on the ring of integers $R$ of a nonarchimedean local field $F$ of characteristic zero, and investigate its properties. Remarkably,  the spectrum of the spectral triple operator is closely related to the roots of a $q$-hypergeometric function. We also study a non compact version of this construction for the C$^*$-algebra of continuous functions on $F$, vanishing at infinity. 
\end{abstract}

\maketitle
\section{Introduction}

The spectral aspect of noncommutative geometry \cite{C1}, \cite{C2} consists of studying the topology and geometry of noncommutative spaces through spectral invariants of an associated geometric operator. The abstract definition of a first-order elliptic operator is given by the concept of a spectral triple. While there are variations of the definition of a spectral triple depending on the context, we use the following for our purpose: a spectral triple for a $C^*$-algebra $A$ is a triple $(\mathcal A, \mathcal H, \mathcal D)$ where $\mathcal H$ is a Hilbert space on which $A$ acts by bounded operators, i.e. there exists a representation $\rho: A \rightarrow B(\mathcal H)$, $\mathcal A$ is a dense $^*$ - subalgebra of $A$, and $\mathcal D$ is an unbounded self-adjoint operator in $\mathcal H$ satisfying: 

(1) for every $a \in \mathcal A$, the commutator $[\mathcal D, \rho(a)]$ is bounded,

(2) for every $a \in \mathcal A$, $\rho(a)(1+ \mathcal D ^2)^{-1/2}$ is a compact operator.
\newline If the algebra $A$ is unital, condition (2) is reduced to:

(2)$^{\prime}$ the resolvent $(1+ \mathcal D^2)^{-1/2}$ is a compact operator.
\newline Moreover,  $(\mathcal A, \mathcal H, \mathcal D)$ is said to be an even spectral triple if there is a bounded operator $\gamma$ on  $\mathcal H$ with $\gamma= \gamma^*$, $\gamma^2=1$ such that $\gamma$ commutes with the representation $\rho$ and anti-commutes with the operator $ \mathcal D$.

In our two previous papers \cite{KMR}, \cite{KRS} we constructed and studied a natural, even spectral triple for  $C(\mathbb Z_p)$, the algebra of continuous functions on the space of $p$-adic integers. Topologically, the space of $p$-adic integers is a Cantor set; however our spectral triple is very different from those that were defined on general Cantor sets by \cite{C1} (chapter 14), \cite{CI}, and \cite{BP}, in the sense that we use the arithmetic of $\mathbb Z_p$ for the construction of our spectral triple. Nonetheless, we follow a similar starting point as \cite{BP}; through Michon's correspondence \cite{M} we associate a tree to a Cantor set equipped with a nonarchimedean metric.  Our operator $\mathcal D$ uses the ring structure of $\mathbb Z_p$ in its definition, and is a form of discrete differentiation similar to many common examples of spectral triples.

Nonarchimedean local fields of characteristic zero are finite extensions of the field of $p$-adic numbers. Alternatively, a nonarchimedean local field $F$ is a locally compact normed field of characteristic zero, where the norm $|\cdot|_p: F \rightarrow \mathbb R_{\geq 0}$ satisfies the  strong triangle inequality $|x+y|_p \leq \max\{|x|_p, |y|_p\}$. The ring of integers of $F$ is defined to be
\begin{equation*}
R:=\{x\in F \; : \; |x|_p \leq 1\},
\end{equation*}
which is a compact metric space, in fact, a Cantor set \cite{C}.

The purpose of this paper is to extend the construction and the results of the two previous papers \cite{KMR} and \cite{KRS} to nonarchimedean local fields. Here we consider two cases; the compact case of $C(R)$, the unital C$^*$-algebra of continuous functions on the ring of integers which is analogous to the situation in \cite{KMR}, \cite{KRS} and the noncompact $C_0(F)$, the non-unital C$^*$-algebra of continuous functions on $F$ vanishing at infinity, which was not considered in our previous papers.

The spectral triples we study in this paper are constructed using associated trees and derivatives on trees. In the compact case, we associate to $R$ equipped with the $p$-adic norm, the tree denoted by $(V_R, E_R)$, where the set of vertices, $V_R$,  consists of balls in $R$.  We then take the Hilbert space of weighted $\ell^2$ functions, $H^R:= \ell^2(V_R,\omega)$,  where the weight function $\omega$ is given by $\omega(v)=$volume$(v)$ for a vertex $v\in V_R$, and consider the action in $H^R$ of the forward derivative $D^R$ on the tree. The representation $\rho_R: C(R) \rightarrow B(H^R)$ of the algebra $C(R)$ on $H^R$  is simply given by the multiplication by the value of $f \in C(R)$ at the preferred center of the ball $v\in V_R$.  We prove that the operator $D^R$ is invertible and $(D^R)^{-1}$ is compact but it is not a Hilbert Schmidt operator unless $R=\mathbb Z_p$. Then we show that the commutator $[D^R, \rho_R(f)]$ is bounded if and only if $f$ belongs to $\mathcal A_R$, the algebra of Lipschitz functions on $R$ (recall that a function $f\in C(R)$ is called Lipschitz if there is a positive constant  $M$ such that $|f(x)-f(y)| \leq M |x-y|_p$). In particular, these results imply that $(\mathcal A_R, \mathcal H^R= H^R \oplus H^R, \mathcal D^R)$, where 
\begin{equation*}
\mathcal D^R= \left(
\begin{array}{cc}
0 & D^R \\
(D^R)^{\ast} & 0
\end{array}\right),
\end{equation*}
is an even spectral triple for $A=C(R)$. Additionally, we prove that the spectral seminorm $L_{\mathcal D^R}(f)= \|[\mathcal D^R,f]\|$ is equivalent to the Lipschitz seminorm (recall that the Lipschitz seminorm of  a function $f$ is $\sup_{x\neq y} \frac{|f(x)-f(y)|}{|x-y|_p}$). We then describe the spectrum of $(D^R)^*D^R$ and some properties of the corresponding zeta function;  the eigenvectors of $(D^R)^*D^R$ turn out to be composed of $q$-hypergeometric functions \cite{ABC}.

To construct a spectral triple for $C_0(F)$ we proceed the same way as in the compact case but using the full (bi-infinite) tree $(V_F, E_F)$ of all balls in $F$, see also \cite{S}.  Again we verify that the forward derivative $D^F$ on the tree is invertible and we study compactness of $\rho_F(f)(D^F)^{-1}$ where $\rho_F: C_0(F) \rightarrow B(H^F)$ is a suitably chosen representation. This allows us to exhibit a dense subalgebra $\mathcal A_F$ in $C_0(F)$ that gives an even spectral triple $(\mathcal A_F, \mathcal H^F= H^F \oplus H^F, \mathcal D^F)$ where 
\begin{equation*}
\mathcal D^F= \left(
\begin{array}{cc}
0 & D^F \\
(D^F)^{\ast} & 0
\end{array}\right).
\end{equation*}

The paper is organized as follows; in section 2 we motivate the construction of a spectral triple by briefly discussing  standard construction of spectral triples for $C(S^1)$ and $C_0(\mathbb R)$. Section 3 discusses basic notations, definitions, $p$-adic Fourier analysis and  trees associated to local fields.  In section 4 the forward derivative operators are introduced on these trees and their invertibility and compactness are discussed. The construction of the spectral triples for $C(R)$ and $C_0(F)$ are presented in section 5 along with a comparison of the distances  induced by the spectral seminorms and the Lipschitz seminorms. Section 6 explores the spectrum of $(D^R)^*D^R$. Here we prove that the eigenvalues of $(D^R)^*D^R$ are roots of the $q$-hypergeometric function $_1\phi_1\left(\substack {0\\q};q,\lambda \right)$. Moreover, we discuss some properties of $((D^R)^*D^R)^{-1}$ and the analytic continuation of the zeta function associated with $(D^R)^*D^R$.


\section{Examples of Spectral Triples}

To motivate the construction of the spectral triples in this paper we recall two standard examples of spectral triples for a compact metric space ($S^1$) and a noncompact space ($\mathbb R$).

{\bf Example 1}: Let $A=C(S^1)$ be the unital  C$^*$-algebra of continuous functions on the unit circle and consider the Hilbert space  $H:= L^2(S^1)$. There is a natural representation $\rho: A \rightarrow B(H)$ given by $\rho(f) \phi(x)=f(x)\phi(x)$ for $f\in A$, $\phi \in L^2(S^1)$. Now consider the action of the operator $D= \frac 1i \frac d{d \theta}$ on the Hilbert space $H$; since $D e^{in\theta}= n e^{in\theta}$, the spectrum of $(1+D^2)^{-1/2}$ consists of  $\left\{\frac 1{\sqrt{1+n^2}} \; :\; n \in \mathbb Z\right\}$.  Therefore, the operator $(1+D^2)^{-1/2}$ is compact and in fact belongs to $(1+\epsilon)$ -Schatten class for $\epsilon >0$. The following algebra: 
\begin{equation*}
\mathcal A= \left\{ f\in C(S^1) \; : \;\sup_{\theta}|f'(\theta)| < \infty| \right\},
\end{equation*}
 is a dense sub-algebra of $A$, and has the property that the commutator $[D, \rho(f)]$ is bounded if and only if $f\in \mathcal A$. Thus the triple $(\mathcal A, H, D)$ is a spectral triple for the C$^*$-algebra $C(S^1)$.

{\bf Example 2}: A similar construction yields a spectral triple for  $A= C_0(\mathbb R)$,  the non-unital C$^*$-algebra of continuous functions on $\mathbb R$ vanishing at infinity.
The algebra $A$ can be represented on the Hilbert space $H=L^2(\mathbb R)$ by multiplication operators; the representation $\rho: A \rightarrow B(H)$ is given by $\rho(f)\phi(x)=f(x)\phi(x)$, where $f\in A$, $\phi \in L^2(\mathbb R)$. Next consider the operator $D= \frac 1i \frac d{dx}$ on $H$. The commutator $[D, \rho(f)]$ is bounded if and only if $\sup_{x\in \mathbb R} |f'(x)|<\infty$. The resolvent $(1+D^2)^{-1/2}$ is not a compact operator, it has a continuous spectrum. However, computing the Hilbert-Schmidt norm of $\rho(f)(1+D^2)^{-1/2}$ we see that
\begin{equation*}
\left\| \rho(f)(1+D^2)^{-1/2} \right\|_{HS}^2 = \int|f(x)|^2 dx \int\frac1{1+k^2}\frac{dk}{2\pi}.
\end{equation*}
 Thus, if we let $\mathcal A$ be the dense $^*$-subalgebra of $A$ given by 
\begin{equation*}
\mathcal A= \left\{ f \in C_0(\mathbb R) \;:\; f\in L^2(\mathbb R), \; \sup_{x\in \mathbb R} |f'(x)| <\infty \right\},
\end{equation*}
it follows that $ \rho(f)(1+D^2)^{-1/2}$ is a compact operator for every $f\in \mathcal A$, and $(\mathcal A, H, D)$ is a spectral triple for the algebra $A$.


\section{Notation and Construction}

\subsection{Local fields}
The main object of study in this paper is a nonarchimedean local field $F$ of characteristic zero. In other words, $F$ is a finite (algebraic) extension of the field of $p$-adic numbers $Q_p$ for some $p$. The $p$-adic absolute value on $\mathbb Q_p$ extends uniquely to an absolute value $|\cdot|$ on $F$ satisfying the properties:
\begin{enumerate}
\item $|x|=0$ if and only if $x=0$,
\item $|xy|= |x||y|$,
\item $|x+y| \leq \max\{|x|, |y|\}$ (nonarchimedean property),
\end{enumerate}
for all $x,y \in F$. The field $F$ is a locally compact space with respect to the topology induced by the norm $|\cdot|$. This norm is normalized so that it coincides with the usual $p$-adic norm on $\mathbb Q_p$. The ring of integers $R$ of the field $F$ is defined as 
\begin{equation*}
R=\{x \in F \; : \; |x|\leq 1\}.
\end{equation*}
The unique maximal ideal of $R$ is 
\begin{equation*}
P=\{x \in F \; : \; |x|< 1\},
\end{equation*}
and the finite field $R/P$, called the residue class field of $F$, is a finite extension of $\mathbb F_p= \mathbb Z/ p\mathbb Z \cong \mathbb Z_p/p\mathbb Z_p $ (the residue class field of $\mathbb Q_p$).  If $\left[ R/P  : \mathbb F_p \right]=f$ then $R/P \cong \mathbb F^{p^f}$. Since the $p$-adic norm is discrete, the ideal $P$ is a principal ideal, hence $P=(\pi)$ for some number $\pi \in P$. We call $\pi$ a local uniformizer of $R$. Alternatively, any element $\pi \in P$ is called a uniformizer if it has the property $|\pi|= \sup_{x\in P} |x|$. Let $S:=\{s_i\} \subset R$ be a set of representatives, which includes zero, of the residue field $R/P$.  If $a \in F$, then $a=\pi^n \epsilon$ for some $n \in \mathbb Z$ and unit $\epsilon$ (i.e. $|\epsilon| =1$). Consequently, $a$ has the representation: 
\begin{equation}\label{padic_series}
a=\sum_{k=k_0}^{\infty} a_k \pi^k \;\; \textnormal{ where } a_k \in S.
\end{equation}

The value group, i.e., the subgroup of $\mathbb R_{>0}$ consisting of the values of $|\cdot|: F^{\times} \rightarrow \mathbb R_{>0}$ consists of $\{p^{j/e}  :  j \in \mathbb Z\}$. The positive integer $e$ for which $|\pi|= p^{-1/e}$ is called the ramification index of $F$ over $\mathbb Q_p$. In fact, $[F : \mathbb Q_p]=ef$ and the extension $F$ is said to be unramified when $e=1$ and totally ramified when $f=1$. \cite{R}  

\subsection{Harmonic analysis on $p$-adic fields \cite{RV}}

We choose a character $\chi: F \rightarrow S^1$ with the property that, if $\chi(xy)=1$ for every $x\in R$ then $y\in R$ as well. The map $x \mapsto \chi_x(y)=\chi(xy)$ for $x\in F$ is an isomorphism between $F$ and $\widehat{F}$, the character group of $F$.  
%
Moreover, two such characters $\chi_x$ and $\chi_y$ coincide on $R$ if and only if $(x-y) \in R$. Consequently, the character group of $R$ is given by $\widehat{R}= F/R$. Since $R$ is a compact abelian topological group, $\widehat R $ is a discrete abelian group which is an extension of the Pr\"ufer $p$-group $\mathbb Q_p / \mathbb Z_p$. 

Let $dx$ be the additive Haar measure on $F$ normalized so that $\int_R dx=1$. Notice that if $B_n(y)$ is the ball described by $B_n(y)=\{x \in F \; : \; |x-y| \leq p^{-n/e}\}$ then the volume of this ball is $p^{-nf}$.

We let $\mathcal E(F)$ denote the space of test functions on F, i.e., functions on $F$ that are locally constant with compact support. The space of locally constant functions on $R$ is denoted by $\mathcal E(R)$. The spaces of distributions on $F$ and $R$ are the spaces of linear functionals on  $\mathcal E(F)$ and $\mathcal E(R)$ respectively, and are denoted by $\mathcal E^*(F)$ and $\mathcal E^*(R)$. Note that such linear functionals are automatically continuous.

Since  $F$ and $\widehat{F}$ are isomorphic, the corresponding spaces of test functions and the spaces of distributions are identified. The space of test functions $\mathcal E(\widehat R)$  is the space of functions on $\widehat R$ which are zero almost everywhere. The space of distributions $\mathcal E^*(\wh R)$ can be identified with the space of all functions on $\widehat R$.

The Fourier transform of a test function $\phi$ on $F$ is the function $\wh\phi$ on  $F$ given by:
\begin{equation*}
\wh\phi(a)=\int_{F}\phi(x)\overline{\chi_a(x)}\,dx.
\end{equation*}
The Fourier transform gives an isomorphism between $\mathcal E(F)$ and $\mathcal E(\wh{F})$, and, by duality, between 
$\mathcal E^*(F)$ and $\mathcal E^*(\wh{ F})$.

The Fourier transform of a test function $\phi$ on $R$ is the function $\wh\phi$ on  $\wh R$ given by the formula above on classes of $a$ in $\wh R=F/R$.
It can be easily verified that, for a ball  $B$ with radius r,  we have $\int_B{\chi_a(x)}\,d_px=0$ whenever $|a|\ge 1/r$. Consequently, only finite number of Fourier coefficients of a locally constant function are nonzero. The Fourier transform gives an isomorphism between $\mathcal E(R)$ and $\mathcal E(\wh R)$. The inverse Fourier transform is:
\begin{equation*}
\phi(x)=\sum_{[a]\in \wh R}\wh\phi([a])\chi_a(x).
\end{equation*}

If $T\in \mathcal E^*(R)$ is a distribution on $R$ then its Fourier transform is the function $\wh T$ on $\wh R$ defined by:
\begin{equation*}
\wh T([a])=T\left(\overline{\chi_a(x)}\right).
\end{equation*}
The Fourier transform is an isomorphism between $\mathcal E^*(R)$ and $\mathcal E^*(\wh R)$, with inverse given by:
\begin{equation*}
T=\sum_{[a]\in \wh{\mathbb Z}_p}\wh T([a])\chi_a(x).
\end{equation*}
The formal sum above makes distributional sense because test functions on $\wh R$ are non zero only at a finite number of points.

As usual, the distributional Fourier transform $T\mapsto \wh T$ gives a Hilbert space isomorphism:
\begin{equation*}
L^2(R, d_px) \cong  l^2(\wh R).
\end{equation*}


\subsection{Trees associated to local fields}

We can associate a weighted rooted tree to $R$ and $F$  via a version of the Michon's correspondence \cite{M}, see also \cite{S}. The corresponding trees will be denoted by $(V^R, E^R)$ and  $(V^F, E^F)$ respectively. We first describe how to obtain the tree corresponding to $R$. The tree corresponding to $F$ will then be obtained using a similar construction.

The set $V^R$ representing the vertices of the tree $(V^R, E^R)$ consists of all balls in $R$. The set of vertices has the natural decomposition $V^R= \cup_{n=0}^{\infty} V_n^R$ where $V_n^R$ consists of balls in $R$ of radius $p^{-n/e}$. We say there is an edge $e=(v,v')$ between two vertices $v, v' \in V^R$ if  $v\in V_n$, $v'\in V_{n+1}$, and  $v'\subset v$.  Thus the set of edges $E^R$ consists of unordered pairs of vertices. 

The description of the tree corresponding to $(V^F, E^F)$ is similar.  The set $V^F$ also has the natural decomposition $V^F=\cup_{n\in \mathbb Z} V_n^F$ where $V_n^F$ is the set of balls in $F$ with radius $p^{-n/e}$, and there is an edge between two vertices in precisely the same situation as for the tree $(V^R, E^R)$. The tree $(V^R, E^R)$ is a subtree of  $(V^F, E^F)$.

If $x$ and $x'$ are two points inside the same ball of radius $p^{-n/e}$, then $x-x' \in \pi^n R$. Consequently $V_n^R \cong R/ \pi^n R$. Thus $V_n^R$ is a finite set with $p^{nf}$ elements. We also have $V_n^F \cong F/ \pi^n R$ which is an infinite set.  

The decomposition (\ref{padic_series}) enables us to identify a preferred center inside each ball. 
Define the following sets:  
\begin{equation*}
X_n^R:=\{x\in R:\ x=\sum_{k=0}^{n-1} x_k \pi^k\}\ \textnormal{and}\ X_n^F:=\{x\in F:\ x=\sum_{k=k_0}^{n-1} x_k \pi^k \text{ for some } k_0 \in \mathbb Z\}.
\end{equation*}
Since $R/ \pi^n R \cong X_n^R$, we see that there is exactly one $x\in X_n^R$ inside each ball in $R$ of radius $p^{-n/e}$.  
As a result, we can parametrize the tree $V^R$ using $\left\{ (n,x) \; \vert \; n \in \mathbb Z_{\geq 0}, x \in X_n^R\right\}$. Similarly, the parametrization of $V^F$ can be done using $\left\{ (n,x) \; \vert \; n \in \mathbb Z, x \in X_n^F\right\}$.

If $v \in V_n^R$ with center $x$ and $v' \in V_{n+1}^R$ with center $x'$, then there is an edge between them if and only if $x'-x =s \pi^{n+1}$ for some $s\in S$. Hence  $(V^R, E^R)$  is a regular tree with deg$(v)=p^f+1$ for each $v \in V^R$. Moreover, it is a subtree of $(V^F, E^F)$.

The weight $w(v)$ of a vertex $v \in V^R$ (or $V^F$) is defined as $w(v)=p^{-nf}$, the volume of the ball.  

Consider the following Hilbert spaces $H^R$ and $H^F$, consisting of weighted $\ell^2$ functions on the vertices $V^R$ and $V^F$ respectively:
\begin{equation*}
\begin{aligned}
H^R&= \ell^2(V^R, w)= \left\{ \phi : V^R \rightarrow \mathbb C \; : \sum_{v\in V^R} |\phi(v)|^2w(v) < \infty \right\}\\
H^F&= \ell^2(V^F, w)=  \left\{ \phi : V^F\rightarrow \mathbb C \; : \sum_{v\in V^F} |\phi(v)|^2w(v) < \infty \right\}.
\end{aligned}
\end{equation*}

The decompositions of the vertex sets induce a natural decomposition of the above two Hilbert spaces as follows: 
\begin{equation*}
H^R= \bigoplus_{n \in \mathbb Z_{\geq 0}} \ell^2(V_n^R, p^{-nf})\;\; \textnormal{  and } \;\;H^F= \bigoplus_{n \in \mathbb Z} \ell^2(V_n^F, p^{-nf}).
\end{equation*}

To distinguish between the arguments for the function $\phi$ and its Fourier transform, we introduce the dual tree denoted by $\widehat{V^R}\cong V^R$. The identification of the vertex set $V_n^R$ with the finite set $R/\pi^n R$ of $p^{nf}$ elements implies that, the Fourier transform of a function $\phi =\{\phi_n\} \in \ell^2(V^R, w)$, is the usual discrete Fourier transform on each $V_n^R$, i.e. it is the function $\widehat{\phi}_n \in \ell^2(\widehat{V_n^R})$ given by:
\begin{equation*}
\widehat{\phi}_n(y)= \sum_{x \in X_n^R} \phi_n(x) \chi\left(-\frac{xy}{\pi^n}\right)p^{-nf}.
\end{equation*}
The inverse Fourier transform is:
\begin{equation*}
\phi_n(x)= \sum_{y \in X_n^R} \widehat{\phi}_n(y) \chi\left(\frac{xy}{\pi^n}\right).
\end{equation*}
Moreover, we have the Parseval's identity: 
\begin{equation*}
p^{-nf} \sum_{x \in X_n^R} |\phi_n(x)|^2 =\sum_{y \in X_n^R}|\widehat{\phi}_n(y)|^2
\end{equation*}
hence the Fourier transform gives an isomorphism between the Hilbert spaces: $$\ell^2(V_n^R, p^{-nf}) \cong \ell^2(\widehat{V_n^R}).$$

Similarly, for $\phi_n \in \ell^2 (V_n^F, p^{-nf})$, its Fourier transform is the function $\widehat{\phi}_n\in L^2(\pi^{-n}R, dx)$ defined by:
\begin{equation*}
\widehat{\phi}_n(\xi)= \sum_{x \in X_n^F} \phi_n(x) \chi\left(-x\xi\right)p^{-nf} \; \textnormal{ for }\; \xi \in \pi^{-n}R.
\end{equation*}
The inverse Fourier transform is:
\begin{equation*}
\phi_n(x)= \int_{\pi^{-n}R} \widehat{\phi}_n(\xi) \chi\left(x\xi \right) d\xi.
\end{equation*}
Again, the Parseval's identity
\begin{equation*}
p^{-nf} \sum_{x \in X_n^F}|\phi_n(x)|^2= \int_{\xi \in \pi^{-n}R} |\widehat{\phi}_n(\xi)|^2 d\xi
\end{equation*} 
holds and the Fourier transform induces an isomorphism: $$\ell^2(V_n^F, p^{-nf}) \cong L^2(\pi^{-n}R, dx).$$

The algebras $C(R)$ and $C_0(F)$ have natural representations in the Hilbert spaces $H^R$ and $H^F$ respectively: the functions are represented as multiplication operators by the values of the centers of the balls, for balls that do not contain zero. For balls that do contain zero, it is more natural to choose nonzero centers whose norms go to infinity as the radius of the balls increases. More precisely, we consider the following representations:
\begin{equation*}
\begin{aligned}
\rho_R&: C(R) \rightarrow \mathcal B(H^R)\\
(\rho_R(a) \phi)_n(x)&=a(x)\phi_n(x) \;\;\textnormal{ for } a\in C(R), x\in X_n^R, \;\;\; \textnormal{ and } 
\end{aligned}
\end{equation*}

\begin{equation*}
\begin{aligned}
\rho_F&: C_0(F) \rightarrow \mathcal B(H^F)\\
(\rho_F(a) \phi)_n(x)&=\begin{cases}
a(x)\phi_n(x) & \textnormal{ for } a\in C_0(F), 0 \neq x\in X_n^F\\
a(\pi^{-n})\phi_n(0) & \textnormal{ for } a\in C_0(F), 0 =x\in X_n^F.
\end{cases}
\end{aligned}
\end{equation*}

The sets $\bigcup_{n\in \mathbb Z_{\geq 0}}X_n^R$ and $\bigcup_{n\in \mathbb Z}X_n^F$ are  dense in $R$ and $F$ respectively. Moreover, their norms are, $||\rho_R(a)||= \sup_{x\in R} |a(x)|$ and $||\rho_F(a)||= \sup_{x\in F} |a(x)|$,  hence the representations are faithful.
\bigskip


\section{A forward derivative on the tree.}\label{derivative_section}

We consider the following operator on the trees $(V^R, E^R)$ and $(V^F, E^F)$, analogous to the operator introduced in our previous paper \cite{KMR}:
\begin{equation*}
(D \phi)_n(x)= p^{n/e} \left(\phi_n(x)-\frac 1{p^f}\sum_{s \in S} \phi_{n+1}(x+ s\pi^n) \right).
\end{equation*}
The versions of $D$ acting on maximal domains in Hilbert spaces $\ell^2(V^R, w)$ and $\ell^2(V^F, w)$ will be denoted by $D^R$ and $D^F$ respectively. That is, if dom$(D^R)=\{\phi\in \ell^2(V^R, w):\ D\phi\in\ell^2(V^R, w)\}$ and dom$(D^F)=\{\phi\in \ell^2(V^F, w):\ D\phi\in\ell^2(V^F, w)\}$ are the maximal domains, then $D^R:= D\vert_{\text{dom}(D^R)}$ and $D^F:= D\vert_{\text{dom}(D^F)}$.

Using the Fourier transform of $\phi_n$, we can write the operators $D^R$ as:

\begin{equation*}
(D^R \phi)_n(x)= p^{n/e} \left(\sum_{y\in X_n^R} \widehat{\phi}_n(y) \chi(\pi^{-n}xy) - \frac 1{p^f} \sum_{s \in S} \sum_{z \in X_{n+1}^R}\widehat{\phi}_{n+1}(z) \chi(\pi^{-(n+1)}(x+ s\pi^n)z) \right).
\end{equation*}
Using the orthogonality of characters,
\begin{equation*}
\frac 1{p^f} \sum_{s\in S}\chi(sz\pi^n)= \begin{cases}
1 & ;\ \pi^n z\in R\\
0 & ;\textnormal{ otherwise,}
\end{cases} 
\end{equation*}
the above formula for $D$ can be written as:
\begin{equation*}
(D^R \phi)_n(x)= p^{n/e} \left(\sum_{y\in X_n^R} \widehat{\phi}_n(y) \chi(\pi^{-n}xy) - \sum_{y \in X_{n+1}^R}\widehat{\phi}_{n+1}(\pi y) \chi(\pi^{-n}x y) \right).
\end{equation*}
Thus the Fourier transform $\widehat D^R$ of $D^R$ is the operator defined in $\ell^2(\widehat{V_n^R})$ by:
\begin{equation*}
(\widehat D^R \widehat \phi)_n(y)= p^{n/e} \left(\widehat{\phi}_n(y) -\widehat{\phi}_{n+1}(\pi y) \right).
\end{equation*}
Similarly, the operator $D^F$ written using Fourier transform is:
\begin{equation*}
(D^F \phi)_n(x)= p^{n/e} \left(\int_{\pi^{-n}R} \widehat{\phi}_n(\xi) \chi(x\xi) d\xi - \int_{\pi^{-n}R}\widehat{\phi}_{n+1}(\xi) \chi(x\xi) d\xi\right).
\end{equation*}
Based on the above formula, we use $\widehat D^F$ to denote the operator defined as
\begin{equation*}
(\widehat D^F \widehat \phi)_n(\xi)= p^{n/e} \left(\widehat{\phi}_n(\xi) -\widehat{\phi}_{n+1}(\xi) \right), \;\textnormal {where } \xi \in \pi^{-n}R\subset \pi^{-(n+1)}R.
\end{equation*}
As before, the operators $\widehat D^R$ and $\widehat D^F$ will be considered on their maximal domains in Hilbert spaces $\ell^2(\widehat{V_n^R})$ and  $L^2(\pi^{-n}R, dx)$.

The formal adjoint of $\widehat D^R$ is given by the formula:
\begin{equation*}
((\widehat D^R)^* \widehat {\psi}))_n(y)= \begin{cases}
p^{n/e} \widehat {\psi}_n(y) & ; \pi \nmid y \\
p^{n/e} \left(\widehat{\psi}_n(y) -p^{-1/e}\widehat{\psi}_{n-1}(\pi^{-1} y) \right) & ; \pi \mid y.
\end{cases}
\end{equation*}

We now conveniently reparametrize the tree using the following one-to-one correspondence between the sets $\{(n,x) \;: \; n\in \mathbb Z_{\geq 0}, x \in X_n^R \}$ and $\{(l,g) \; : \; l \in \mathbb Z_{\geq 0}, g \in F/R\}$.

Given $(n,x)$ with $n \in \mathbb Z_{> 0}$ and $0\ne x \in X_n^R $, let $l$ be the unique nonnegative integer such that $|x|= p^{-l/e}$ and $g:= \frac x{\pi^{n}} \in F/R$. For $x=0$ and any $n \in \mathbb Z_{\geq 0}$  we let $l=n$ and $g=0$. Thus we have the correspondence $(n,x) \mapsto (l,g)$.

Conversely, given $(l,g)$ such that $l \in \mathbb Z_{\geq 0}$ and $0\ne g \in F/R$, let $m\in \mathbb Z_{> 0}$ be  such that $|g|= p^{m/e}$. Then we make the association $(l,g) \mapsto (m+l, \pi^{m+l}g)$. For $g=0$ and any $l \in \mathbb Z_{\geq 0}$  we have $n=l$ and $x=0$. Thus, we have established a one-to-one correspondence between the above two sets of parameters. 

The operators $D^R$ and $(\widehat D^R)^*$ have simpler expressions in terms of the new parameters:
\begin{equation*}
\begin{aligned}
\widehat {D^R} \widehat \phi(l,g)&= p^{(m+l)/e} \left[\widehat \phi (l,g)-\widehat \phi (l+1, g)\right]\\
((\widehat D^R)^* \widehat {\phi}))(l,g)&= p^{(m+l)/e} \left[\widehat \phi (l,g)-p^{-1/e}\widehat \phi (l-1, g) \right]\; ; \; \;\widehat \phi(-1, g)=0,
\end{aligned}
\end{equation*}
if $g\ne 0$. Notice that in this case the above formula can be written as:
\begin{equation*}
(\widehat D^R\widehat \phi) (l,g)= |g|p^{l/e} \left[\widehat \phi (l,g)-\widehat \phi (l+1, g)\right].
\end{equation*}
Additionally, we have:
\begin{equation*}
(\widehat D^R\widehat \phi) (l,0)= p^{l/e} \left[\widehat \phi (l,0)-\widehat \phi (l+1, 0)\right].
\end{equation*}
Consequently, the Hilbert space $\widehat H^R$ and the operator $ \widehat D^R$ have the following decompositions: 

\begin{equation*}
\widehat H^R=\bigoplus_{g \in F/R} \widehat H_g^R ;\;\;\;  \widehat D^R= \bigoplus_{g \in F/R} \widehat D_g^R,
\end{equation*}
where $\widehat D_g^R:= \widehat D^R \vert_{\widehat H_g^R}$, and $\widehat H_g^R\cong \ell^2(\mathbb Z_{\geq 0})$.

The theorem below establishes the main analytical properties of the operator $D^R$.

\begin{theo}
$D^R$ is invertible and $(D^R)^{-1}$ is a compact operator. Moreover, $(D^R)^{-1}$ is a Hilbert-Schmidt operator if and only if $F=\mathbb Q_p$.
\end{theo}

\begin{proof} This theorem can be easily proved using the component operators in the above decomposition for $\widehat D^R$. When $g\neq 0$, we can write:
\begin{equation*}
(\widehat D_g^R) h(l)= |g| p^{l/e} \left[h(l)-h(l+1)\right].
\end{equation*}
Given $(\widehat D_g^R) h(l)=k(l)$ for $k(l)\in \ell^2(\mathbb Z_{\geq 0})$, we can solve for $h(l)$ to obtain: 
\begin{equation*}
(\widehat D_g^R)^{-1} k(l)= h(l)= \frac 1{|g|} \sum_{j \geq l} p^{-j/e}k(j).
\end{equation*}
Additionally, \[(\widehat D_0^R)^{-1} k(l)= \sum_{j \geq l} p^{-j/e}k(j).\]
The limits in the above sums are obtained by requiring $h(l) \in \ell^2(\mathbb Z_{\geq 0})$. Then the Hilbert-Schmidt norm of $(\widehat D_g^R)^{-1} $ is:
\begin{equation}\label{compact_components}
\left\|(\widehat D_g^R)^{-1}\right\|_{HS}^2= \sum_{l=0}^{\infty}\sum_{j\geq l}\frac 1{|g|^2} p^{-2j/e} =  \frac 1{|g|^2}\cdot \frac1{(1-p^{-2/e})^2}.
\end{equation}
Consequently, we have: $$\left\| (\widehat D_g^R)^{-1}\right\| \leq \left\|(\widehat D_g^R)^{-1}\right\|_{HS} = \frac 1{|g|} \cdot \frac 1{(1-p^{-2/e})} \rightarrow 0\ \ {\textrm as\ \ }|g| \rightarrow \infty.$$
Since $(D^R)^{-1}$ is a direct sum of compact operators whose norms go to zero by \eqref{compact_components}, it follows that $(D^R)^{-1}$ is also compact.  

Next we compute the Hilbert-Schmidt norm of $(\widehat D^R)^{-1}$:
\begin{equation*}
\left\|(\widehat D^R)^{-1}\right\|_{HS}^2= \sum_{g\in F/R}\frac 1{|g|^2} \frac 1{(1-p^{-2/e})^2}= \frac 1{(1-p^{-2/e})^2} \sum_{m=0}^{\infty} \frac{c(m)}{p^{2m/e}},
\end{equation*}
where $c(m)= \#\{g\in F/R \; : \; |g|=p^{m/e}\}$. In fact, since
\begin{equation*}
c(m)= \begin{cases}
1 & ;  \textnormal{ if }  m=0\\
p^{mf}(1- \frac 1{p^f}) & ; \textnormal{ otherwise,}
\end{cases}
\end{equation*}
we get: \begin{equation*}
\left\|(\widehat D^R)^{-1}\right\|_{HS}^2=\frac 1{(1-p^{-2/e})^2} \left[ 1+(1-\frac 1{p^f}) \sum_{m=1}^{\infty} \frac 1{p^{m(\frac{2-ef}{e})}}\right],
\end{equation*}
and the sum on the right hand side is convergent if and only if $ef=n=1$. i.e.  $(\widehat D^R)^{-1}$ is a Hilbert-Schmidt operator if and only if $F=\mathbb Q_p$.

\end{proof}
\bigskip

%
 
In the non-compact case, the operator $D^F$ has again no kernel, but its range is only dense in $H^F$. This is best seen from the formula in Fourier transform:
\begin{equation}\label{DFinverse}
((\widehat D^F)^{-1} \widehat \psi)_n(\xi)= \sum_{k=n}^\infty p^{-k/e} \widehat{\psi}_k(\xi).
\end{equation}
which becomes ill-defined as $n\to-\infty$.
We have, however, the following result, relevant for our spectral triples construction.

\begin{theo}\label{compact_op}
Assume $a: F \rightarrow \mathbb C$ is a continuous function such that $|a(x)| = \mathcal O \left(\frac{1}{1+|x|^{\alpha}}\right)$ where $\alpha >1, \alpha > ef/2$. Then, $\rho_F(a)(D^F)^{-1}$ is a compact operator. 
\end{theo}

\begin{proof}
Applying inverse Fourier transform to the formula (\ref{DFinverse}), we get
\begin{equation*}
(\rho_F(a)(D^F)^{-1}\phi)_n(x)= a_n(x)\sum_{k=n}^{\infty} \sum_{y \in X_k^F} p^{-k/e}p^{-kf} \phi_k(y)\int_{\pi^{-n}R} \chi((x-y)\xi) d\xi,
\end{equation*}
where
\begin{equation*}
a_n(x)= \begin{cases}
a(x) & ;\ x\neq 0\\
a(\pi^{n}) & ;\ x=0.
\end{cases}
\end{equation*}
Hence
\begin{equation*}
(\rho_F(a)(D^F)^{-1}\phi)_n(x)= \sum_{k=n}^{\infty} \sum_{\substack{y \in X_k^F \\ (x-y)\in \pi^nR}}  p^{-k/e}p^{f(n-k)} a_n(x) \phi_k(y).
\end{equation*}


To prove compactness of $\rho_F(a)(D^F)^{-1}$ we use an approximation argument. For $t>0$ let 
\begin{equation*}
b_t(k)= \begin{cases}
1 & ; k<0\\
\frac 1{1+tp^{\alpha k/e}} & ; k\geq 0,
\end{cases}
\end{equation*}
and consider the operator $\rho_F(a)(D^F)^{-1} b_t$ given by
\begin{equation*}
\begin{aligned}
\left((\rho_F(a)(D^F)^{-1}b_t)\phi\right)_n(x)&=\sum_{\substack {{n \leq k < \infty}} }\sum_{\substack{y \in X_k^F \\ (x-y)\in \pi^nR}}  p^{-k/e}p^{f(n-k)} a_n(x) b_t(k)\phi_k(y).
\end{aligned}
\end{equation*}
We first show that $(\rho_F(a)(D^F)^{-1}b_t) $ is a Hilbert-Schmidt operator, by estimating its norm:
\begin{equation*}
\left\|(\rho_F(a)(D^F)^{-1}b_t)\right \|_{HS}^2=\sum_{n \leq k <0} \sum_{y \in X_k^F} \sum_{\substack{x \in X_n^F \\ (x-y)\in \pi^nR}}  {p^{\frac{-2k}e + 2f(n-k)}}|a_n(x)|^2.
\end{equation*}

\rem In all of the computations below, the constant $C$ may vary from line to line.

We split the above sums into four parts. First, the contribution from $x=0$ and $k<0$ is less than:
\begin{equation*}
\begin{aligned}
& C \sum_{n \leq k <0} \sum_{\substack{y \in X_k^F \\ y\in \pi^nR}}  \frac {p^{\frac{-2k}e + 2f(n-k)}}{(1+ p^{-\alpha n/e})^2} = \; C \sum_{k<0} p^{\frac{-2k}e -fk} \sum_{n \leq k} \frac{p^{fn}}{(1+ p^{-\alpha n/e})^2}\\
& \leq C \sum_{k<0} \frac{p^{\frac{-2k}e} }{(1+ p^{-\alpha k/e})^2(1-p^{-f})} < \infty. 
\end{aligned}
\end{equation*}

Similarly, for $x=0$ and $k\geq 0$, we bound the contribution to the Hilbert-Schmidt norm by:
\begin{equation*}
\begin{aligned}
&C \sum_{\substack{k \geq 0 \\ k \geq n}} \sum_{\substack{y \in X_k^F \\ y\in \pi^nR}}  \frac {p^{\frac{-2k}e + 2f(n-k)}}{(1+ p^{-\alpha n/e})^2(1+tp^{\alpha k/e})^2} = C \sum_{k \geq 0} \frac {p^{\frac{-2k}e -fk}}{(1+tp^{\alpha k/e})^2} \sum_{n \leq k} \frac{p^{fn}}{(1+ p^{-\alpha n/e})^2}\\
&\leq C \sum_{k\geq 0} \frac {p^{\frac{-2k}e -fk}}{(1+tp^{\alpha k/e})^2 (1+p^{-\alpha k/e})^2}\sum_{n \leq k} p^{fn} =C \sum_{k \geq 0} \frac {p^{\frac{-2k}e }}{(1+tp^{\alpha k/e})^2 (1+p^{-\alpha k/e})^2} < \infty.
\end{aligned}
\end{equation*}

If $x \neq 0$ and $k<0$, then we estimate the contribution by:
\begin{equation*}
\begin{aligned}
&C \sum_{n \leq k <0} \sum_{y \in X_k^F} \sum_{\substack{0\neq x \in X_n^F \\ (x-y)\in \pi^nR}}  \frac {p^{\frac{-2k}e + 2f(n-k)}}{(1+ |x|^{\alpha})^2} =\\
& = C \sum_{k<0} \sum_{n\leq k} \sum_{l<n}  \frac {p^{\frac{-2k}e + f(n-k)}}{(1+ p^{\frac{-l\alpha}e})^2}  p^{f(n-l)} \leq C \sum_{k<0}  p^{\frac{-2k}e - fk}\sum_{n \leq k} p^{2fn}\sum_{l<n}p^{\frac{l(2\alpha -ef )}{e}} = \\ &= C  \sum_{k<0}  p^{\frac{-2k}e - fk}\sum_{n \leq k} p^{\frac{n(2\alpha +ef )}{e}} = \sum_{k<0} p^{\frac{2k(\alpha -1)}{e}} < \infty.
\end{aligned}
\end{equation*}

Finally, in the case when $x \neq 0$ and $k\geq 0$,  the contribution is bounded by:
\begin{equation*}
\begin{aligned}
&C \sum_{k \geq 0}\sum_{n \leq k} \sum_{y \in X_k^F} \sum_{\substack{0\neq x \in X_n^F \\ (x-y)\in \pi^nR}}  \frac {p^{\frac{-2k}e + 2f(n-k)}}{(1+ |x|^{\alpha})^2 (1+tp^{\alpha k/e})^2} \\
&= C \sum_{k\geq 0} \sum_{n\leq k} \sum_{l<n}  \frac {p^{\frac{-2k}e + f(n-k)}}{(1+ p^{\frac{-l\alpha}e})^2} \frac{p^{f(n-l)}}{(1+tp^{\alpha k/e})^2} \leq C \sum_{k\geq 0}  \frac {p^{\frac{-2k}e -fk}}{(1+ tp^{\alpha k/e)^2}} \sum_{n \leq k} p^{\frac {n(2\alpha + ef)}{e}} \\
&= \sum_{k \geq 0} \frac {p^{\frac{2k(\alpha -1)}e}}{(1+ tp^{\alpha k/e})^2}\leq C  \sum_{k\geq 0}  \frac{p^{\frac{2k(\alpha -1)}{e}}}{t^2p^{\frac{2\alpha k}e}} < \infty.
\end{aligned}
\end{equation*}
This proves that $(\rho_F(a)(D^F)^{-1}b_t)$ is a Hilbert-Schmidt operator.

Next we  show that $\left\|(\rho_F(a)(D^F)^{-1}b_t) - \rho_F(a)(D^F)^{-1}\right\|\rightarrow 0$  as $t \rightarrow 0^+$.  To estimate the operator norm of $\rho_F(a)(D^F)^{-1}(b_t-1)$ we use Schur-Young inequality \cite{HS}:
\begin{equation*}\begin{aligned}
&\left\|(\rho_F(a)(D^F)^{-1}(b_t-1)\right\|^2 \leq \\ &\bigg( \sup_{k}\sup_{y\in X_k^F}\sum_n \sum_{x\in X_n^F}|K(k,y; n,x)|\bigg)\bigg( \sup_{n}\sup_{x\in X_k^F}\sum_k \sum_{y\in X_n^F}|K(k,y; n,x)|\bigg),
\end{aligned}
\end{equation*}
where $K(k,y; n,x)$ is the integral kernel of the operator $\rho_F(a)(D^F)^{-1}(b_t-1)$. Below we estimate each of the terms on the right-hand side  of this inequality.

The first factor is:
 \begin{equation*}
\sup_{k\geq 0} \sup_{y\in X_k^F} \sum_{\substack{n \leq k}} \sum_{\substack{ x\in X_n^F \\ (x-y)\in \pi^nR}} p^{\frac{-k}e + f(n-k)}|a_n(x)| \left|\frac{1}{1+tp^{\alpha k/e}} -1 \right|.
\end{equation*}

First notice that given $y\in X_k^F$ there is only one $x\in X_n^F$, with $n\leq k$, such that $(x-y)\in \pi^nR$, namely $x=y$ mod $\pi^n$. Moreover for this $x$ we have $|x|=|y|$.
The part with $y \neq 0$ in the first factor is consequently bounded by:
 \begin{equation*}
C \sup_{k \geq 0} p^{\frac{-k(1+f)}e } \left|\frac{1}{1+tp^{\alpha k/e}} -1 \right| \sup_{0\ne y\in X_k^F} \sum_{n \leq k} p^{fn}\frac 1{1+|y|^{\alpha}}, 
\end{equation*}
where we have used the hypothesis of the theorem on the size of $a(x)$.  Since 
\begin{equation*}
\sup_{y\in X_k^F}\frac 1{1+|y|^{\alpha}}=1,
\end{equation*}
we can estimate the last expression as:

 \begin{equation*}
C \sup_{k \geq 0} p^{\frac{-k}e -fk} \left(1- \frac{1}{1+tp^{\alpha k/e}}\right) \frac{p^{fk}}{(1-p^{-f})} =  \frac{C}{(1-p^{-f})} \sup_{k \geq 0} p^{-k/e} \frac{tp^{\alpha k/e}}{1+ tp^{\alpha k/e}} \rightarrow 0,
\end{equation*}
as $t \rightarrow 0^+$.
When $y=0$, the contribution is bounded by:
 \begin{equation*}
 \begin{aligned}
&C \sup_{k\geq 0} p^{\frac{-k}e -fk} \left(1- \frac{1}{1+tp^{\alpha k/e}}\right)  \sum_{n \leq k} \frac{p^{fn}}{1+tp^{-\alpha n/e}} \\ 
&\leq  \frac{1}{(1-p^{-f})} \sup_{k \geq 0} p^{-k/e} \frac{tp^{\alpha k/e}}{1+ tp^{\alpha k/e}} \rightarrow 0 \textnormal { ( as } t \rightarrow 0^+).
\end{aligned}
\end{equation*}

The second factor in the Schur-Young inequality is:
\begin{equation*}
\sup_{n} \sup_{x\in X_n^F} \sum_{\substack{k \geq n \\ k \geq 0}} \sum_{\substack{ y\in X_k^F \\ (x-y)\in \pi^nR}} p^{\frac{-k}e + f(n-k)}|a_n(x)| \left|\frac{1}{1+tp^{\alpha k/e}} -1 \right|.
\end{equation*}
Since the first factor in that inequality is already going to zero as $t$ goes to zero, all we need to show is that the second factor is bounded.
Notice that, given $x\in X_n^F$, the number of $y\in X_k^F$ with $n\leq k$, such that $(x-y)\in \pi^nR$ is equal to $p^{(k-n)f}$.
Consequently, we can estimate the second factor by:
 \begin{equation*}
C \sup_{n} \sup_{x \in X_n^F} \sum_{\substack{k \geq n \\ k \geq 0}} p^{\frac{-k}e} |a_n(x)|,
\end{equation*}
where we estimated: 
\begin{equation*}
\left(1- \frac{1}{1+tp^{\alpha k/e}}\right)\leq 1.
\end{equation*}
The function $|a_n(x)|$ is uniformly bounded in $x$ and $n$ by assumptions, hence  the above expression is bounded by a uniform constant for every $n \in \mathbb Z$. 
Thus we have:
\begin{equation*}
\left\|(\rho_F(a)(D^F)^{-1}b_t) - \rho^F(a)(D^F)^{-1}\right\|\rightarrow 0
\end{equation*}
 as $t \rightarrow 0^+$ and $\rho_F(a)(D^F)^{-1}$ is a compact operator.
\end{proof}

\rem In fact, with the above assumption on $a(x)$,  the operator $\rho_F(a)(D^F)^{-1}$ is a Hilbert-Schmidt operator for every $a(x)$ if and only $ef=1$, i.e. when $F=\mathbb Q_p$, as can be deduced from the proof above.
\bigskip

\section{Construction of the spectral triples.}
\subsection{A Spectral triple for $C(R)$.}

We first describe a spectral triple for the C$^*$-algebra $C(R)$ and study its properties. Much of the discussion here is a straightforward generalization of the spectral triple construction in \cite{KMR} and will be referred to whenever needed.
 
Let $\mathcal A_R$ be the algebra of Lipschitz functions on $R$, a dense $^*$-subalgebra of $C(R)$.  We use $L_1(a)$ to denote the Lipschitz seminorm of $a\in \mathcal A_R$ defined by: 
\begin{equation}\label{Lip_sem}
L_1(a)=\sup_{\substack{x\neq y \\ x,y \in R}} \frac{|a(x)-a(y)|}{|x-y|_p} \in [0, \infty].
\end{equation}
Also let $\mathcal H^R$ be the direct sum: $\mathcal H^R := H^R \oplus H^R$. Define $\mathcal D^R: \mathcal H^R \rightarrow \mathcal H^R$ as: 
\begin{equation*}
\mathcal D^R= \left(
\begin{array}{cc}
0 & D^R \\
(D^R)^{\ast} & 0
\end{array}\right).
\end{equation*}

The algebra $C(R)$ can be represented in $\mathcal H^R$  using $\Pi_R: C(R) \rightarrow \mathcal B(\mathcal H^R)$ given by  $\Pi_R= \rho_R \oplus \rho_R$. This representation is even, faithful and non-degenerate since $X_n^R$ is dense in $R$.

\begin{theo}\label{sptr_for_R}
The triple $(\mathcal A_R, \mathcal H^R, \mathcal D^R)$ is an even spectral triple.
\end{theo}

\begin{proof}
We first show that the commutator $[\mathcal D^R, \Pi_R(a)]$ is bounded for every $a \in \mathcal A_R$. The formula for the commutator is:
\begin{equation*}
[D^R\rho_R(a)g]_n(x)- [\rho_R(a)D^Rg]_n(x)= \frac 1{p^f} \sum_{i=0}^{p^f -1} \frac{[a_n(x)-a_{n+1}(x+ s_i\pi^n)]}{p^{-n/e}}   g_{n+1}(x+ s_i \pi^n).  
\end{equation*}        
In the case $x \neq 0$ we get:     
\begin{equation*}
\begin{aligned}
\left | (D^R \rho_R(a)g)_n(x)-  (\rho_R(a)D^Rg)_n(x) \right | &\leq \frac 1{p^f} \sum_{i=1}^{p^f -1} \frac{\left|a(x)-a(x+ s_i\pi^n)\right |}{p^{-n/e}}   \left|g_{n+1}(x+ s_i \pi^n) \right|\\
& \leq \frac{L_1(a)}{p^f} \sum_{i=1}^{p^f-1} \left | g_{n+1}(x+s_i \pi^n)\right|.
\end{aligned}
\end{equation*}     
When $x=0$ we estimate:
\begin{equation*}
\begin{aligned}
&\left |(D^R\rho_R(a)g)_n(0)- (\rho_R(a)D^Rg)_n (0)\right |\leq \\
 &\leq \frac 1{p^f} \frac{\left| a(\pi^n) - a(\pi^{n+1})\right|}{p^{-n/e}} |g_{n+1}(0)| + \frac 1{p^f} \sum_{i=1}^{p^f -1} \frac{\left|a(\pi^n)-a(s_i\pi^n)\right |}{p^{-n/e}}   \left|g_{n+1}( s_i \pi^n) \right|\\
& \leq \frac{L_1(a)}{p^f} \left( |g_{n+1}(0)| +\sum_{i=1}^{p^f -1}\left|g_{n+1}( s_i \pi^n) \right|\ \right).
\end{aligned}
\end{equation*}  
Hence, we get:
\begin{equation*}
\begin{aligned}
\left \| D^R \rho_R(a)g-  \rho_R(a)D^Rg\right \|^2 &\leq \sum_{n=0}^{\infty} \sum_{x\in X_n^R} p^{-nf}  \frac{L_1(a)^2}{p^{2f}} \left( \sum_{i=1}^{p^f-1} \left | g_{n+1}(x+s_i \pi^n)\right| \right)^2 \\
& \leq L_1(a)^2 \sum_{n=0}^{\infty} \sum_{x\in X_n^R} p^{-nf-f} \left | g_{n+1}(x+s_i \pi^n)\right|^2,
\end{aligned}
\end{equation*}    
where in the second line we used the fact that $\left(\sum_{i=1}^n a_n\right)^2 \leq n \sum_{i=1}^n a_i^2$ for $a_i \geq 0$. 
Re-labelling the sum we get:
\begin{equation*}
\left \|D^R \rho_R(a)g-  \rho_R(a)D^Rg\right \|^2 \leq  L_1(a)^2 \sum_{n=0}^{\infty} \sum_{x\in X_{n+1}^R} p^{-(n+1)f} \left | g_{n+1}(x)\right|^2 \leq L_1(a)^2 \|g\|^2.
\end{equation*}    
Taking the adjoint in the formula $\left\| \left[D^R, \rho_R(a)\right]\right\| \leq L_1(a)$, we also obtain the inequality $\left\|\left[(D^R)^*, \rho_R(a)\right]\right\| \leq L_1(a)$. This shows that the commutator $[\mathcal D^R, \Pi_R(a)]$ is bounded for every $a\in \mathcal A_R$.    
 
Since $(D^R)^{-1}$ is a compact operator, by functional calculus it follows that $(1+ (\mathcal D^R )^2)^{-1/2}$ is a compact operator. Thus, $(\mathcal A_R, \mathcal H^R, \mathcal D^R)$ is an even spectral triple.
                                                                                                                                                                          
\end{proof}

We will now investigate the distance functions induced by two seminorms: the Lipschitz seminorm $L_1: C(R) \rightarrow [0, \infty]$, defined in (\ref{Lip_sem}), and the spectral seminorm $L_{\mathcal D^R}:C(R) \rightarrow [0, \infty]$ defined by $L_{\mathcal D^R}(a)= \left\| \left[\mathcal D^R, \Pi_R(a) \right]\right\| =  \left\| \left[ D^R, \rho_R(a) \right]\right\| $, where unbounded operators have infinite norm. These two seminorms induce the following two distance functions on $R$: 

\begin{equation*}
\begin{aligned}
\textnormal{dist}_1(x,y)&=\sup_{a\in C(R)} \left\{ \left|a(x)-a(y)\right| \; : \; L_1(a) \leq 1\right\}\\
\textnormal{dist}_{\mathcal D^R}(x,y)&=\sup_{a\in C(R)} \left\{ \left|a(x)-a(y)\right| \; : \; L_{\mathcal D^R}(a) \leq 1\right\},
\end{aligned}
\end{equation*}
for $x,y \in R$.
It is a general fact that the metric $\textnormal{dist}_1$ is equal to the usual $p$-adic metric on $R$, see Proposition 5.1 in \cite{KMR}.
In order to understand the other metric, we need to control the norm of the commutator $\left[ D^R, \rho_R(a) \right]$. This is achieved with the help of the following lemma, the proof of which is similar to the proof of Lemma 5.2 in \cite{KMR}.

\begin{lem}\label{norm_form}
For $a \in C(R)$, \begin{equation*}
\begin{aligned}
L_{\mathcal D^R}(a)^2=   \sup_{\substack{n\in \mathbb Z_{\geq 0} \\ { 0 \neq x\in X_n^R}}} & \left\{ \frac 1{p^f} \sum_{i=1}^{p^f-1} \frac{\left|a(x)-a(x+s_i \pi^n)\right|^2}{p^{-2n/e}},   \frac 1{p^f}  \sum_{i=1}^{p^f-1} \frac{\left|a(\pi^n)-a(\pi^{n+1})\right|^2}{p^{-2n/e}}, \right. \\
&\left. \frac 1{p^f} \sum_{i=1}^{p^f-1} \frac{\left|a(\pi^n)-a(s_i \pi^n)\right|^2}{p^{-2n/e}}\right \}.
\end{aligned}
\end{equation*}
\end{lem}

We can use the formula for the spectral seminorm in the above lemma to prove the following theorem.

\begin{theo} \label{norm_comp}
For every $a\in \mathcal A_R$ we have:
\begin{equation*}
\frac{p^{1/e}-1}{2p^{1/e} \sqrt{p^f}} L_1(a) \leq L_{\mathcal D^R}(a) \leq \sqrt{\frac{p^f-1}{p^f}}L_1(a).
\end{equation*}
\end{theo}

\begin{proof}
Using the Lipschitz property of $a$ we see that:
\begin{equation*}
\frac 1{p^f} \sum_{i=1}^{p^f-1} \frac{\left| a_n(x)- a_{n+1} (x+ s_i \pi^n)\right| ^2}{p^{-2n/e}} \leq \frac{p^f-1}{p^f} L_1(a)^2.
\end{equation*}
Taking the supremum over $n$ and $x$ we obtain the inequality: $L_{\mathcal D^R}(a)^2 \leq  \frac{p^f-1}{p^f} L_1(a)^2$.

To prove the other side of the inequality, we first observe that for all $a \in \mathcal A^R, n\geq 0$ and $x\in R$ we have:
\begin{equation*}
\frac 1{p^f} \sum_{i=1}^{p^f-1} \frac{\left| a_n(x)- a_{n+1} (x+ s_i \pi^n)\right| ^2}{p^{-2n/e}} \leq L_{\mathcal D^R}(a)^2,
\end{equation*}
which implies that:
\begin{equation*}
 \frac{\left| a_n(x)- a_{n+1} (x+ s_i \pi^n)\right| }{p^{-n/e}} \leq \sqrt{p^f} L_{\mathcal D^R}(a).
\end{equation*}

For any $a \in \mathcal A^R$ and $y \neq z \in R$ (both not zero) with $|y-z|_p= p^{-n/e}, n\geq 0$, there exists $x \in X_n^R$ with $x < \pi^n$ so that $y$ and $z$ have the following $p$-adic representations:
\begin{equation*}
\begin{aligned}
y&=x+ y_0\pi^n + y_1 \pi^{n+1} + \cdots \\
z&=x+ z_0\pi^n + z_1 \pi^{n+1} + \cdots 
\end{aligned}
\end{equation*}
with $y_i, z_i \in S$ and $y_0, z_0 \neq 0$. First we look at the case $x\neq 0$. Telescoping, we obtain:
\begin{equation*}
\begin{aligned}
&\left| a(x+y_0\pi^n + \cdots + y_N\pi^{n+N}) - a(x) \right|\leq \\
&\leq \left| a(x) - a(x + y_0\pi^n) \right| + \left| a(x+y_0\pi^n) - a(x + y_0\pi^n + y_1 \pi^{n+1}) \right|+\cdots + \\
&+\left| a(x+ y_0\pi^n + \cdots + y_{N-1}\pi^{n+N-1}) - a(x+ y_0\pi^n + \cdots+ y_N \pi^{n+N}) \right|. 
\end{aligned}
\end{equation*}
This leads to the estimate:
\begin{equation*}
\begin{aligned}
\left| a(x+y_0\pi^n + \cdots + y_N\pi^{n+N}) - a(x) \right| &\leq \sqrt{p^f} L_{\mathcal D^R}(a) \left[  p^{-n/e} + p^{-(n+1)/e} + \cdots + p^{-(n+N)/e}\right]\\
& =  \sqrt{p^f} L_{\mathcal D^R}(a)p^{-n/e} \frac{1-p^{-(N+1)/e}}{1-p^{-1/e}}.
\end{aligned}
\end{equation*}
Since $a$ is continuous, taking $N \rightarrow \infty$ in the above inequality we obtain:
\begin{equation*}
|a(y)-a(x)| \leq  \sqrt{p^f} L_{\mathcal D^R}(a) \frac{p^{-n/e}}{1-p^{-1/e}}.
\end{equation*}
We have a similar inequality for $|a(z)-a(k)|$. Hence, we get:
\begin{equation*}
|a(y)-a(z)| \leq \frac{2\sqrt{p^f}}{1-p^{-1/e}} |y-z|_p L_{\mathcal D^R}(a).
\end{equation*}
Since this is true for every $y,z \in R$, we have:
\begin{equation*}
L_1(a) \leq \frac{2\sqrt{p^f}}{1-p^{-1/e}} L_{\mathcal D^R}(a),
\end{equation*}
from which the other side of the inequality follows.

If $x =0$, then we estimate similarly:
\begin{equation*}
\begin{aligned}
&\left| a(y_0\pi^n + \cdots + y_N\pi^{n+N}) - a(\pi^n) \right| \\
&\leq \left| a(\pi^n) - a(y_0\pi^n) \right| + \left| a(y_0\pi^n) - a(y_0\pi^n + y_1 \pi^{n+1}) \right|+\cdots + \\
&\left| a(y_0\pi^n + \cdots + y_{N-1}\pi^{n+N-1}) - a( y_0\pi^n + \cdots+ y_N \pi^{n+N}) \right|.
\end{aligned}
\end{equation*}
We consider separately two cases $y_0 \neq 0$, and $y_0 =0$. If $y_0 \neq 0$, then the above inequality implies that:
\begin{equation*}
\begin{aligned}
&\left| a(y_0\pi^n + \cdots + y_N\pi^{n+N}) - a(\pi^n) \right| \leq \sqrt{p^f} L_{\mathcal D^R}(a)p^{-n/e}\frac{1-p^{-(N+1)/e}}{1-p^{-1/e}},
\end{aligned}
\end{equation*}
and, by taking the limit as $N \rightarrow \infty$, we get: $$|a(y)- a(x)| \leq \sqrt{p^f} L_{\mathcal D^R}(a)\frac{p^{-n/e}}{1-p^{-1/e}}.$$  It follows that: $|a(y)- a(z)| \leq  \frac{2\sqrt{p^f}}{1-p^{-1/e}} |y-z|_p L_{\mathcal D^R}(a)$, whenever $y,z$ are not both zero.

Now suppose $y_0 =0$ and $y_M \neq 0$ for some $M$. Then, we have:
\begin{equation*}
\begin{aligned}
&\left| a(y_M\pi^{n+M} + \cdots + y_N\pi^{n+N}) - a(\pi^n) \right| \\
&\leq \left| a(\pi^n) - a(y_M\pi^{n+M}) \right| + \left| a(y_M\pi^{n+M}) - a(y_M\pi^{n+M} + y_{M+1} \pi^{n+M+1}) \right|+\cdots + \\
&\left| a(y_M\pi^{n+M} + \cdots + y_{N-1}\pi^{n+N-1}) - a( y_M\pi^{n+M} + \cdots+ y_N \pi^{n+N}) \right|.
\end{aligned}
\end{equation*}
The first term in the sum on the right hand side can be estimated as:
\begin{equation*}
\begin{aligned}
\left| a(\pi^n) - a(y_M\pi^{n+M}) \right| \leq & \left| a(\pi^n) - a(\pi^{n+1}) \right| + \left| a(\pi^{n+1}) - a(y_M\pi^{n+2}) \right| +\cdots \\
& \cdots + \left| a(\pi^{n+M-1}) - a(\pi^{n+M}) \right| .\\
\end{aligned}
\end{equation*}
Consequently, we arrive at the estimate:
\begin{equation*}
\begin{aligned}
&\left| a(y_M\pi^{n+M} + \cdots + y_N\pi^{n+N}) - a(\pi^n) \right| \\
&\leq \sqrt{p^f} L_{\mathcal D^R}(a) p^{-n/e}\left[ p^{\frac{-(n+N)}e} + p^{\frac{-(n+N-1)}e} + \cdots + p^{\frac{-(n+M)}e} + p^{\frac{-(n+M-1)}e} + \cdots + p^{\frac{-n}e}\right].\\
\end{aligned}
\end{equation*}
Taking the limit as $N \rightarrow \infty$ as before, we can once again establish the inequality \[|a(y)- a(z)| \leq  \frac{2\sqrt{p^f}}{1-p^{-1/e}} |y-z|_p L_{\mathcal D^R}(a),\] whenever $y,z$ are not both zero.

Now we consider the case where one of $y$ or $z$ is zero, say $z=0$. In this case, take a sequence of integers $z_k \in R$ such that $z_k \rightarrow 0$ as $k\rightarrow \infty$ and $|z_k|< |y|$. Then, we obtain:
\begin{equation*}
|a(y)-a(0)| \leq 2 |a(y)- a(z_k)| + |a(z_k) - a(0)| \leq \frac{2 \sqrt {p^f}}{1-p^{-1/e}} L_{\mathcal D^F}(a)|y| + |a(z_k) - a(0)|
\end{equation*}
for every $k$. Since $|a(z_k)- a(0)| \rightarrow \infty$ as $k \rightarrow \infty$, we have:
\begin{equation*}
|a(y)-a(0)| \leq \frac{2 \sqrt {p^f}}{1-p^{-1/e}} |y| L_{\mathcal D^F}.
\end{equation*}
Summarizing, in all of the above subcases of case $x=0$ it follows that $L_1(a) \leq \frac{2 \sqrt {p^f}}{1-p^{-1/e}} L_{\mathcal D^F}$. Therefore, the inequality in the theorem is established.
\end{proof}

As an immediate consequence of Lemma \ref{norm_comp} we have the following corollary.
\begin{cor}\label{Lipschitz}\  
\begin{enumerate}
\item Consider the algebra given by $C^1(R):=\{ a\in C(R): \|[D, \rho_R(a)]\| < \infty\}$. Then $C^1(R) = \mathcal A_R$. 
\item The commutant of $C(R)$, i.e., the algebra $\{a \in C(R)\; : \; [D, \rho_R(a) ] =0\}$, is the trivial algebra consisting of constant functions.
\end{enumerate}
\end{cor}

\begin{proof}
The proof is similar to that of Corollary 5.4 in \cite{KMR}.
\end{proof}

From the above results now follows that the usual $p$-adic metric on $R$ is equivalent to the metric induced my the spectral seminorm.

\begin{theo} For every $x,y \in R$ we have:
\begin{equation*}
\frac{p^{1/e}-1}{2p^{1/e} \sqrt{p^f}} |x-y|_p \leq \textnormal{dist}_{\mathcal D^R}(x,y) \leq \sqrt{\frac{p^f-1}{p^f}}|x-y|_p.
\end{equation*}
\end{theo}

\begin{proof}
From Theorem \ref{norm_comp}, it is clear that:
\begin{equation*}
\textnormal{dist}_{\mathcal D^R}(x,y) =\sup_{a: L_{\mathcal D^R (a) }\neq 0} \frac {|a(x)-a(y)|}{L_{\mathcal D^R (a)}}.
\end{equation*}
It follows from Corollary \ref{Lipschitz} that $L_1(a)=0$ if and only if $L_{\mathcal D^R}=0$. Consequently, we can write:
\begin{equation*}
 \frac {|a(x)-a(y)|}{L_{\mathcal D^R }(a)}=  \frac {|a(x)-a(y)|}{L_1(a)} \frac{L_1(a)}{L_{\mathcal D^R }(a)},
\end{equation*}
and the result now follows from Theorem  \ref{norm_comp}.
\end{proof}

Additionally, we have the following straightforward analog of Theorem 5.6 in \cite{KMR}.

\begin{theo}
The pair $(C(R), L_{\mathcal D^R})$ is a compact spectral metric space.
\end{theo}

\subsection{A Spectral triple for $C_0(F)$.}

Let $\mathcal A_F$ be the following subalgebra of $C_0(F)$:
\begin{equation*}
\mathcal A_F = \left\{ a\in C_0(F) :  a \textnormal{ is Lipschitz continuous,  }   |a(x)|= \mathcal O \left(\frac{1}{1+|x|^{\alpha}}\right) ,\  x\in F \right\},
\end{equation*}
where $\alpha >1, \alpha > \frac{ef}2$.
It is an easy consequence of the Stone-Weierstrass theorem that $\mathcal A_F$ is a dense $*$-subalgebra of $C_0(F)$. We use $L_2: C_0(F) \rightarrow [0, \infty]$ to denote the Lipschitz norm of a function $a \in C_0(F)$.

Once again, we take $\mathcal H^F = H^F \oplus H^F$ and define $\mathcal D^F: \mathcal H^F \rightarrow \mathcal H^F$ to be:
\begin{equation*}
\mathcal D^F= \left(
\begin{array}{cc}
0 & D^F \\
(D^F)^{\ast} & 0
\end{array}\right).
\end{equation*}

The representation $\Pi_F: C_0(F) \rightarrow \mathcal B(\mathcal H^F)$ defined by $\Pi_F= \rho_F \oplus \rho_F$ is even, faithful and non-degenerate. Consequently, we have all the ingredients to define a spectral triple for $C_0(F)$.
\begin{theo}
The triple $(\mathcal A_F, \mathcal H^F, \mathcal D^F)$ is an even spectral triple for the C$^*$algebra $C_0(F)$. 
\end{theo}

\begin{proof}
The computation involved in showing that the commutator $[\mathcal D^F, \Pi_F(a)]$ is bounded for every $a \in \mathcal A_F$ is essentially the same as the computation in the proof of Theorem \ref{sptr_for_R} showing that $[\mathcal D^R, \Pi_R(a)]$ is bounded, and thus we will not repeat it. Then by Theorem \ref{compact_op} it follows that the above triple is an even spectral triple.
\end{proof}

Similarly to what we did with the operator $\mathcal D^R$, we can compare the metrics induced by the Lipschitz seminorm $L_2: C_0(F) \rightarrow [0, \infty]$ and $L_{\mathcal D^F}: C_0(F)  \rightarrow [0, \infty]$ defined by $L_{\mathcal D^F}= \left\|[ \mathcal D^F, \Pi_F(a)]\right\| = \left\| [ D^F, \rho_F(a)]\right\| $.  The distance functions induced by these norms on $F$ are, respectively:
\begin{equation*}
\begin{aligned}
\textnormal{dist}_2(x,y)&=\sup_{a\in C_0(F)} \left\{ \left|a(x)-a(y)\right| \; : \; L_2(a) \leq 1\right\}\\
\textnormal{dist}_{\mathcal D^F}(x,y)&=\sup_{a\in C_0(F)} \left\{ \left|a(x)-a(y)\right| \; : \; L_{\mathcal D^F}(a) \leq 1\right\}.
\end{aligned}
\end{equation*}

\begin{lem} For $a \in C_0(F)$ we have:
\begin{equation*}
\begin{aligned}
L_{\mathcal D^F}(a)^2=   \sup_{\substack{n\in \mathbb Z \\ { 0 \neq x\in X_n^F}}} & \left\{ \frac 1{p^f} \sum_{i=1}^{p^f-1} \frac{\left|a(x)-a(x+s_i \pi^n)\right|^2}{p^{-2n/e}},   \frac 1{p^f}  \sum_{i=1}^{p^f-1} \frac{\left|a(\pi^n)-a(\pi^{n+1})\right|^2}{p^{-2n/e}}, \right. \\
&\left. \frac 1{p^f} \sum_{i=1}^{p^f-1} \frac{\left|a(\pi^n)-a(s_i \pi^n)\right|^2}{p^{-2n/e}}\right \}.
\end{aligned}
\end{equation*}
\end{lem}

\begin{proof}
The proof is similar to proof of Lemma \ref{norm_form}.
\end{proof}

Using the above spectral seminorm formula, once again we can compare the norms $L_2$ , $L_{\mathcal D^F}$ and their induced metrics. We summarize these results in the following theorem.

\begin{theo} For every $a\in \mathcal A_{F}$ and $x,y \in F$ the following are true:
\begin{enumerate}
\item $\frac{p^{1/e}-1}{2p^{1/e}\sqrt{p^f}}L_2(a) \leq L_{\mathcal D^F}(a) \leq \sqrt{\frac{p^f -1}{p^f}} L_1(a)$,
\item $\frac{p^{1/e}-1}{2p^{1/e}\sqrt{p^f}} |x-y|_p \leq \textnormal{dist}_{\mathcal D^F} (x,y) \leq \sqrt{\frac{p^f -1}{p^f}} |x-y|_p$.
\end{enumerate}
\end{theo}

\bigskip
\section{Spectral properties.}

In section \ref{derivative_section} we saw that the operators $\widehat D^R$ and $(\widehat D^R)^*$ assume simpler expressions in the parametrization with coordinates $l,g$.  We will need the following notation:
\begin{equation*}
\begin{aligned}
\widehat D_0^R \widehat \phi(l)& := p^{\frac le} \left(\widehat \phi (l) - \widehat \phi(l+1)\right)\;\; \textnormal{ and }\\
(\widehat D_0^R)^*\widehat \phi(l) & :=  p^{\frac le} \left( \widehat \phi (l) -\frac 1{p^{1/e}}\widehat \phi(l-1)\right).
\end{aligned}
\end{equation*}
For $|g|=p^{m/e}$, this allows us to write: 
\begin{equation*}
\widehat D_g^R= p^{m/e} \widehat D_0^R \;\;\textnormal { and } \;\; (\widehat D_g^R)^*= p^{m/e}( \widehat D_0^R )^*,
\end{equation*}
which consequently gives the decomposition:
\begin{equation}\label{decom}
(\widehat D^R)^*\widehat D^R = \bigoplus_{g \in F/R} (\widehat D_g^R)^*\widehat D_g^R =  \bigoplus_{g \in F/R}  p^{2m/e}( \widehat D_0^R )^* \widehat D_0^R. 
\end{equation}
The eigenvalue equations for $( D_0^R )^* D_0^R$ are:
\begin{equation}\label{EVE}
\begin{aligned}
(D_0^R )^* D_0^R \phi(0)& = \phi(0) -\phi(1) =\lambda \phi(0)\\
 (D_0^R )^* D_0^R \phi(l)& = p^{\frac{2(l-1)}e} \left[ -p^{\frac 2e}\phi(l+1)+ (1+ p^{\frac 2e}) \phi(l)-\phi(l-1) \right] = \lambda \phi(l),\;\; \textnormal{ for } l\geq 1, 
\end{aligned}
\end{equation}
with $\phi(l) \in \ell^2(\mathbb Z_{\geq 0})$.

\begin{lem}
The eigenvectors  $\phi(l, \lambda)$ of $(D_0^R )^* D_0^R $ satisfy the following:
\begin{enumerate}
\item $\phi(l, \lambda) \in \ell^1(\mathbb Z_{\geq 0})$ and $\phi(l, \lambda)=\sum_{k=1}^\infty c(2k)p^{-2nk/e}$ where 
\begin{equation*}
c(2k)=\left(\frac{-\lambda}{1-p^{-2/e}}\right)^{k-1} \frac{c(2)p^{\frac{k(k-1)}{e}} (p^{2/e}-1)^{k-2}}{(p^{4/e}-1)^2 (p^{6/e}-1)^2 \cdots (p^{(2k-2)/e}-1)^2 (p^{2k/e}-1)},
\end{equation*}
for $ k\geq 2$, and $$c(2)= \left(\frac{\lambda}{1-p^{-2/e}}\right) \sum_{k=0}^\infty \phi(l).$$
\item If $r_n(2N)$ is the remainder:
$$\phi(l,\lambda)=c(2) p^{-2n/e} + c(4) p^{-4n/e} + \cdots + c(2N-2)p^{-(2N-2)n/e} + r_n(2N),$$
 then $r_n(2N) \rightarrow 0$ as $N \rightarrow \infty$ in the $\ell^2$ norm.
\end{enumerate}
\end{lem}

\begin{proof}
Comparing the eigenvalue equations (equations 4.1) in \cite{KMR} and the equations \eqref{EVE} above we see that the only difference is that $p$ is replaced by $p^{1/e}$. Hence the proof of this lemma follows from Lemma 4.1 in \cite{KMR}: 
\end{proof}

Using the above lemma we can now find the spectrum of the operator $(D^R )^* D^R$. Below we describe the eigenvalues of $(D^R )^* D^R$ as the roots of the $q$-hypergeometric function \cite{GR}:
\begin{equation*}
_1\phi_1\left(
\begin{array}{c@{}c@{}c}
\begin{array}{c}
 0\\
 q\\
\end{array} ;&\ q ,&\ z\\
\end{array}\right)=\sum_{n=0}^\infty \frac{1}{(q;q)_n(q;q)_n}(-1)^nq^{\binom{n}{2}}z^n, \;\;\; \textnormal{ with } q=p^{-2/e}.
\end{equation*}

\begin{theo}
Let $\lambda_n$ be the roots of $\lambda \mapsto  _1\phi_1\left(\substack {0\\q};q,\lambda \right)$ with $q=p^{\frac{-2}e}$ arranged in the increasing order. Then the following are true:
\begin{enumerate}
\item $\sigma((D_0^R)^*D_0^R)= \{\lambda_n \;: n \in \mathbb Z_{\geq 0} \}$, and each $\lambda_n$ is a simple eigenvalue.
\item $\sigma((D^R)^*D^R) = \bigcup_{n,m \in \mathbb Z_{\geq 0}} \left\{ p^{2m/e}\lambda_n\right\},$ where, for $m\geq 1$, each eigenvalue occurs with multiplicity $p^{mf}(1-\frac 1{p^f})$.
\end{enumerate}
\end{theo}

\begin{proof}
Showing that eigenvalues of $(D_0^R)^*D_0^R$ are roots of $_1\phi_1\left(\substack {0\\q};q,\lambda \right)$ follows from the proof  of Theorem 4.2 in \cite{KMR}.
Finding the spectrum of $(D^R)^*D^R$ is a straightforward calculation using the decomposition \eqref{decom}. The number of elements $g \in F/R$ with $|g|=p^{m/e}$ for $m\geq 1$ is $p^{mf}-p^{(m-1)f}$, so the multiplicity of each eigenvalue of $(D^R)^*D^R$ is $p^{mf}(1-\frac 1{p^f})$.
\end{proof}

\begin{cor}
$(D^R)^{-1}$ is in $s$-Schatten class for all $s\geq ef$.
\end{cor}

\begin{proof}
By the decomposition \ref{decom} we have:
\begin{equation*}
\left((D^R)^*D^R\right)^{-s} = \bigoplus_{g \in F/R} p^{\frac{-2ms}e} \left((D_0^R)^*D_0^R\right)^{-s},
\end{equation*}
and hence 
\begin{equation*}
\textnormal{Tr}\left((D^R)^*D^R\right)^{-s} = \sum_{n=0}^{\infty}\left(1+ \sum_{m=1}^{\infty} p^{\frac{-2ms}e} (p^{mf} - p^{(m-1)f})\right) \lambda_n^{-s}.
\end{equation*}
The upper and lower bounds for the eigenvalues are discussed in \cite{ABC}, \cite{KMR}. In our case we have the following estimates:
\begin{equation*}
p^{\frac ne} \left( 1- \frac{p^{-2n/e}}{1-p^{-2n/e}}\right) \leq \lambda_n \leq p^{\frac ne}, \;\;\; \textnormal{ for } n \geq 1.
\end{equation*}
Since $1- \frac{p^{-2n/e}}{1-p^{-2n/e}} \geq 1- \frac{1}{p^{2/e}-1}$ for $s>0$, we have:
\begin{equation*}
\sum_{n=1}^{\infty} \lambda_n^{-s} \leq \sum_{n=1}^{\infty} p^{\frac {-ns}e} \left( 1- \frac{p^{-2n/e}}{1-p^{-2n/e}}\right)^{-s}  \leq \sum_{n=1}^{\infty} p^{\frac {-ns}e} \left( 1- \frac{1}{p^{2/e}-1}\right)^{-s} <\infty.
\end{equation*}
 Moreover, we can estimate:
\begin{equation*}
\sum_{m=1}^{\infty} p^{\frac{-2ms}e} (p^{mf} - p^{(m-1)f})= (1-p^{-f}) \sum_{m=1}^\infty p^{m(f-\frac {2s}e)} = \frac{(1-p^{-f})p^{(f-\frac {2s}e)}}{1-p^{f-\frac{2s}e}}, \;\;\; \textnormal{ for } s> \frac{ef}{2}.
\end{equation*}
Consequently, the $s$-th Schatten norm of $(D^R)^{-1}$ is finite whenever $s \geq ef$.
\end{proof}

Let $\zeta_{D^R}$ and $\zeta_{D_0^R}$ be the zeta functions associated with the operators $(D^R)^*D^R$ and $(D_0^R)^*D_0^R$ respectively. We have:
\begin{equation*}
\zeta_{D^R}(s)=  \left(\frac{1-p^{-\frac{2s}e}}{1-p^{f-\frac{2s}e}} \right)\zeta_{D_0^R}(s).
\end{equation*}

\begin{theo} The zeta functions defined above have the following properties:

\begin{enumerate}
\item $\zeta_{D_0^R}$ is holomorphic for Re $s >0$ and can be analytically continued to a meromorphic function for Re $s > \frac{-2}e$.
\item $\zeta_{D^R}$ is meromorphic for Re$s > \frac{-2}e$ with poles at $s=\frac{2 \pi i ke}{\ln p}$ and $s = \frac e2 (f- \frac{2 \pi ik}{\ln p})$, where $k \in \mathbb Z$.
\end{enumerate}
\end{theo}

\begin{proof}
The proof is similar to that of Theorem 5.3 and Corollary 5.4 in \cite{KMR}.
\end{proof}


\begin{thebibliography}{99}

\bibitem{ABC}
L. D. Abreu, J. Bustoz and J.L. Cardoso, The Roots Of the Third Jackson $q$ - Bessel Function, IJMMS 2003:67, 4241-4248 PII. S016117120320613X.

\bibitem{BP}
J. Bellissard and J.C. Pearson, Noncommutative Riemannian Geometry and Diffusion on Ultrametric Cantor Sets, J.  Noncommut. Geo., 3, (2009), 447-480.

\bibitem {C}
J.W.S. Cassels, Local fields. London Mathematical Society Student Texts, 3. Cambridge University Press, Cambridge, 1986. xiv+360 pp. ISBN: 0-521-30484-9; 0-521-31525-5

\bibitem{C1}
A. Connes, Noncommutative Geometry, Academic Press, 1994.

\bibitem{C2}
A. Connes, On the Spectral Characterization of Manifolds, arXiv:0810.2088 [math.OA].

\bibitem{CI}
E. Christensen and C. Ivan, Spectral triples for AF C*-algebras and metrics on the Cantor set, J. Oper. Theory (2006), 17-46.


\bibitem{GR}
G. Gasper and M. Rahman, Basic Hypergeomeric Series, Encyclopedia of Mathematics and its Applications, vol. 35, Cambridge University Press, Cambridge, 1990.


%


%
\bibitem{HS}
P.R. Halmos, V.S. Sunder, {\it Bounded Integral Operators on $L^2$ Spaces}, Springer-Verlag, 1978.

%
%

\bibitem{KMR}
S. Klimek, M. McBride, S. Rathnayake, A $p$-adic Spectral Triple, J. Math. Phys. 55, 113502 (2014).

\bibitem{KRS}
S. Klimek, S. Rathnayake, K. Sakai, A note on spectral properties of the $p$-adic tree, J. Math. Phys. 57, 023512 (2016).

%
%
\bibitem{M}
 G. Michon, Les Cantors reguliers, C. R. Acad. Sci. Paris Ser. I Math., (19), 300, (1985) 673-675.
%
%
 \bibitem{R}
 A. Robert, A Course in $p$-adic Analysis, Springer 2000.

 \bibitem{RV}
 D. Ramakrishnan, R. Valenza, Fourier Analysis on Number Fields, Springer 1999.
 
 \bibitem{S}
 J-P. Serre, Trees, Springer 1980.
%
%
%
%
\end{thebibliography}
\end{document}